\documentclass[a4paper]{article}
\usepackage{amsmath}
\usepackage{amssymb}
\usepackage{amsthm}
\usepackage{authblk}
\newtheorem{lemma}{Lemma}

\title{Intrinsic determination of the criticality of a slow-fast Hopf bifurcation}
\author[1]{Peter De Maesschalck}
\author[2]{Thai Son Doan}
\author[1]{Jeroen Wynen}
\affil[1]{Department of Mathematics, Hasselt University, Hasselt, Belgium}
\affil[2]{Institute of Mathematics, Vietnam Academy of Science and Technology, Hanoi, Vietnam}
\date{May 18, 2020}
\def\grad{\nabla}

\newenvironment{vf}{\left\{\begin{array}{rcl}}{\end{array}\right.}
\def\R{\mathbf{R}}
\newtheorem{theorem}{Theorem}
\newtheorem{proposition}{Proposition}
\theoremstyle{remark}
\newtheorem{remark}{Remark}

\begin{document}

\maketitle

\begin{abstract}
  The presence of slow-fast Hopf (or singular Hopf) points  in slow-fast systems in the plane is often deduced from the shape of a vector field brought into normal form.  It can however be quite cumbersome to put a system in normal form.  In \cite{Mono} an intrinsic presentation of slow-fast vector fields is initiated, showing hands-on formulas to check for the presence of such singular contact points.  We generalize the results in the sense that the criticality of the Hopf bifurcation can be checked with a single formula.  We demonstrate the result on a slow-fast system given in non-standard form where slow and fast variables are not separated from each other. The formula is  convenient since it does not require any parameterization of the critical curve.


\end{abstract}

\section{Introduction}\label{sec:intro}

Slow-fast systems in the plane are capable of undergoing singular bifurcations (or slow-fast bifurcations), like singular Hopf, singular Bogdanov-Takens  bifurcations or even bifurcations of higher codimension.  In comparison to regular bifurcations, these ones are actually $\epsilon$-families of bifurcations, where $\epsilon$ is the singular parameter used for time scale separation.  Typically, the bifurcation diagram shrinks to a point in the singular limit $\epsilon=0$.\medskip

In \cite{sfbt} the notion slow-fast codimension $k$ bifurcation was introduced, $k=1$ corresponding to singular Hopf and $k=2$ to singular Bogdanov-Takens  bifurcation.  The criterium stipulates that a system undergoes such a  bifurcation if it can be brought into a local normal form having a specific shape.  In that respect, it is an indirect way
of identifying the bifurcations.  If one combines this with the knowledge that putting systems into normal form can be computationally challenging, it is clear that there is a need for an intrinsic characterization.  We will first review some results derived in \cite{Mono}, where we intrinsically determine the presence of slow-fast codimension $k$ bifurcations, and then refine the results by providing a formula to see the criticality of a slow-fast Hopf bifurcation.  The formula can easily be implemented in a symbolic manipulator, and could  be useful in studying applications (for example \cite{app}, \cite{app2}, \dots). It can be applied to singular Bogdanov-Takens bifurcations where it determines the criticality of the Hopf bifurcation curve inside the BT bifurcation diagram (see \cite{wech-pdm}, \cite{sfbt}).
\medskip

\begin{remark}
The results extend the coordinate-free approach in \cite{WECH}, in that monograph the author did not derive a formula for the criticality, which is the main point here.
The context of the paper is limited to systems with two variables (eg.~systems on $\R^2$ or more generally on a smooth surface), though the results trivially extend to vector fields on invariant manifolds (like center manifolds) of higher dimensional systems.  (Results on the criticality of singular Hopf bifurcations in higher dimensions can be found in \cite{highdim}; that paper restricts the presentation to the standard form of slow-fast systems.) The results in this paper do not require high degrees of smoothness: smoothness class $C^5$ will for sure amply suffice. We finally remark that we have limited ourselves to computing first Lyapunov quantity of a singular Hopf; higher codimension cases are treated in \cite{birth} for systems in Li\'enard normal form.
 \end{remark}

In Section~\ref{sec:review} we give a short review of planar slow-fast systems, limited to the context that we need here.  In Section~\ref{sec:results} we state the new results, which we prove in Section~\ref{sec:proof}.  We then illustrate the formula on two cases in Section~\ref{sec:examples}.

\section{Review on Slow-fast systems}\label{sec:review}

This section is based entirely on the presentation of slow-fast systems in the monograph \cite{Mono}.
Let $X_{\epsilon,\lambda}$ be a smooth family of vector field on a smooth surface $M$, defined in a neighbourhood of some $p\in M$, and defined for $\epsilon\in [0,\epsilon_0)$ and $\lambda$ in some open neighborhood $\Lambda$ of $\lambda_0$.  The parameter set $\Lambda$ is assumed to be inside some Euclidean space $\R^n$.

The ``\emph{fast  vector field}'' $X_{0,\lambda}$ is supposed to have a  curve of singular points
\[
    S_\lambda = \{ q\in M : F_\lambda(q) = 0 \},
\]
called the ``\emph{critical curve}''. We will ({\bf assumption A1}) assume the curve to be smooth and regular in the sense that it has near $p$ a well-defined tangent (and we assume that $p\in S_{\lambda_0}$. We write
\[
    X_{\epsilon,\lambda} = F_\lambda.Z_\lambda + \epsilon Q_\lambda + O(\epsilon^2).
\]
Here, $F_\lambda$ is a scalar family of functions, $Z_\lambda$ and $Q_\lambda$ are families of vector fields; we may express the vector field on a local chart in $\R^2$ as
\[
    \begin{vf}
        \dot{x} &=& F_\lambda(x,y)Z^1_\lambda(x,y) + \epsilon Q^1_\lambda(x,y) + O(\epsilon^2),\\
        \dot{y} &=& F_\lambda(x,y)Z^2_\lambda(x,y) + \epsilon Q^2_\lambda(x,y) + O(\epsilon^2),
    \end{vf}
\]
but we will avoid using this explicit shape as much as possible.

The vector field $Q_\lambda$ is related but is not equal to the slow vector field: in fact, if we would denote with $\tilde{Q}_\lambda(q)$ the linear projection of $Q_\lambda(q)$ on the tangent line $T_q(S_\lambda)$ (projection along the direction of $Z_\lambda(q)$), then the vector field $q\mapsto \tilde{Q}_\lambda(q)$ would be a vector field along the curve $S_\lambda$, which is better known as the ``\emph{slow vector field}''.  We will not be needing this notion explicitly in this paper, so we refer to \cite{Mono} for relating this intrinsic definition to the more usual version that is based on first bringing the shape of the slow-fast system in standard form or that is based on the construction of center manifolds.\medskip

The principal slow-fast ingredients of the slow-fast family of vector fields can hence be derived from the triplet $(F_\lambda, Z_\lambda, (Q_\lambda)|_{S_\lambda})$.

Assume furthermore that $Z_\lambda(p)\not=0$ ({\bf assumption A2}), which allows to conclude the last part of the next lemma:

\begin{lemma}\cite{Mono}
    The fast-vector field is normally hyperbolic at a point $q$ on the critical curve (eg.~with $F(q)=0$) if and only if $Z_\lambda(F)(q)\not=0$.  It is normally attracting when $Z_\lambda(F)(q)<0$, normally repelling when $Z_\lambda(F)(q)>0$.  Points where $Z_\lambda(F)(q)=0$ are called contact points.  The linearization of the fast vector field at such point is always nilpotent.
\end{lemma}

In the above lemma and at other places in the paper we have identified vector fields $V$ with their Lie derivative $\mathcal{L}_V$ and conveniently write $V(f) = \mathcal{L}_V(f)$.  In a chart, it would mean that $V(f) = V_x\frac{\partial f}{\partial x} + V_y\frac{\partial f}{\partial y}$.\medskip

Under these conditions it is always possible (see \cite{Mono}) to find a local chart of $M$ where the slow-fast vector field is given, up to multiplication by some positive function $w(x,y,\epsilon,\lambda)$ by
\begin{equation}\label{eq:nf}
    \begin{vf}
        \dot{x} &=& y-f_\lambda(x)  \\
        \dot{y} &=& \epsilon( g_\lambda(x) + (y-f_\lambda(x))h_\lambda(x,y) + O(\epsilon))
    \end{vf}
\end{equation}
(We would then have $(F_\lambda,Z_\lambda,Q_\lambda|_{S_\lambda}) = (y-f_\lambda(x),w.\frac{\partial}{\partial x}, w.g_\lambda\frac{\partial}{\partial y}$.)

Our results will deal with a point $p$ that is ({\bf assumption A3}) a contact point for $\lambda=\lambda_0$, eg.~$F_{\lambda_0}(p)=0$ and $Z_{\lambda_0}(F_{\lambda_0})(p)=0$.  We will furthermore assume that the contact point is a generic one in the sense that
\[
    (Z_{\lambda_0})^2(F_{\lambda_0})(p) := Z_{\lambda_0}(Z_{\lambda_0}(F_{\lambda_0}))(p) \not=0.
\]
In the local expression (\ref{eq:nf}) it means that $f_{\lambda_0}'(x_0)=0$ and $f_{\lambda_0}''(x_0)\not=0$.  The genericity implies

\begin{proposition}\label{prop:genericity}
    Assume $X=XZ+\epsilon Q$ satisfy assumptions 1--2.  Let $p$ satisfy assumption 3.  There exists a neighborhood $\Lambda\subset\R^n$ of $\lambda_0$ and a map $p:\Lambda\to M:\lambda\mapsto p_\lambda$ such that $p_{\lambda_0}=p$ and such that $p_\lambda$ is a generic contact point for all $\lambda$:
    \[
        F_\lambda(p_\lambda) = 0,\qquad Z_\lambda(F_\lambda)(p_\lambda) = 0,\qquad Z^2_\lambda(F_\lambda)(p_\lambda)\not=0.
    \]
\end{proposition}
Though the statement is actually proven in \cite{Mono} we will give a short proof here as well, see Remark~\ref{rm:proofpropgen}.\medskip

Different kinds of contact points are distinguished next.  Most commonly, contact points are so-called jump points.  But it can also be a slow-fast Hopf point, a slow-fast BT (Bogdanov-Takens) point or of higher codimension.  The two most common situations can be characterized using Lie derivatives and Lie Brackets.  Recall that a Lie bracket of a vector field, identified with its Lie derivative, is given by
\[
    [V,W]: f\mapsto V(W(f))-W(V(f)).
\]

\begin{lemma}\cite{Mono}
    The point $p_\lambda$ is a \emph{jump point} if and only if $Q_{\lambda}(F_{\lambda})(p)\not=0$.   In the other case it is called a \emph{singular contact point}.  When
    \[
        Q_{\lambda}(F_{\lambda})(p)=0\qquad\text{and}\qquad
        [Z_{\lambda},Q_{\lambda}](F_{\lambda})(p_\lambda)<0
    \]
    the point $p$ is a slow-fast Hopf point.  (When
    $[Z_{\lambda},Q_{\lambda}](F_{\lambda})(p_\lambda)>0
    $
    it is called a singular contact point of index -1, it is a singular saddle.)
\end{lemma}

In the local expression (\ref{eq:nf}) it can easily be checked that jump points $p_\lambda=(x_\lambda,y_\lambda)$ correspond to $g_{\lambda}(x_\lambda)\not=0$ and singular Hopf points to the case $g_{\lambda}(x_\lambda)=0$ and $g_{\lambda}'(x_\lambda)<0$.

\begin{remark}\label{rm:triplet}
The triple $(F_\lambda,Z_\lambda,Q_\lambda)$ is not entirely defined intrinsically; in fact application of an $\epsilon$-dependent  change of coordinates  could lead to changing the triplet into an ``equivalent'' $(F_\lambda,Z_\lambda,Q_\lambda + cZ_\lambda)$  for some scalar function $c$, and moreover the factorization of $X_0$ into $F_\lambda.Z_\lambda$ is defined up to a nonzero scalar scalar function $b$: $(b.F_\lambda,b^{-1}.Z_\lambda,Q_\lambda)$ is an equivalent triplet.  Choosing another factorization does not affect the conditions of the lemma, see~\cite{Mono}.
\end{remark}

In \cite{Mono}, an equivalent criterium for the presence of a slow-fast Hopf point can be found, based on the choice of an area form $\Omega$ on the manifold $M$; that led to a function that was invariant up to some multiplicative factor. We will improve this result to find an actual intrinsic function, independent on the choice of the area form and independent w.r.t.~factorization equivalence mentioned in Remark~\ref{rm:triplet}.

\section{Statement of the results}\label{sec:results}

Recall that under assumption A1--A3, the contact point $p$ for $\lambda_0$ perturbs smoothly to contact points $p_\lambda$ for nearby parameter values $\lambda$.  In the next theorem we define two values that will reveal important to characterize the criticality.

\begin{theorem}\label{thm:intrinsic}
    Let $(M,g)$ be a Riemannian manifold, and let $\Omega$ be the area form associated to the metric $g$, and let also be $\nabla$ the gradient associated to that metric.  Under assumptions A1--A3 (for which it actually suffices to assume $A3$ for $\lambda=\lambda_0$, seen Proposition~\ref{prop:genericity}),
    the function
    \[
        \mathcal{G}_{\lambda} = \Omega(Q_{\lambda}, Z_{\lambda} ) . \Omega (\grad F_{\lambda } ,\grad Z_{\lambda}F_{\lambda} )
    \]
    on $S_{\lambda}$ and the value
    \[
        \left.\mathcal{A}_{\lambda} = \frac{    Z_{\lambda}^3(F_{\lambda})}{ \left(Z_{\lambda}^2(F_{\lambda})\right)^2}\right|_{p_\lambda}
    \]
    are well-defined, are coordinate free notions, independent w.r.t.~factorization equivalence mentioned in Remark~\ref{rm:triplet}.
    Furthermore, $\mathcal{G}_{\lambda_0}$ does not depend  on the choice of the metric.
\end{theorem}

\begin{remark}
    \begin{enumerate}
    \item
    The result in Theorem~\ref{thm:intrinsic} is strong since it allows to compute $\mathcal{G}$ and $\mathcal{A}$ in any coordinate system.  We refer to the example section for details.  In fact, by allowing arbitrary choices of the chart and metric, it is by far the most easy choice to pick the coordinate system in which the system is given with the standard Euclidean metric.
     \item
        The invariant $\mathcal{A}$ could be related to the ``skewness'' of the critical curve: it is determined solely by the fast vector field $X_0$.
        The invariant $\mathcal{G}$ is closely related to the slow vector field, and from the function we will derive a so-called ``skewness'' of the slow vector field.  The two notions together determine the criticality of a Hopf bifurcation as we will see.
     \end{enumerate}
\end{remark}

In order to give a formula for the criticality, we will make use of the intrinsically defined vector field $V_\lambda$ for which
\[
    V_\lambda(F_\lambda)=0,\qquad V_\lambda(Z_\lambda(F_\lambda))=1,
\]
along the critical set.  It is uniquely defined up to an ideal generated by $F$.

\begin{theorem}\label{thm:criticality}
    Assume A1--A3.   The contact point $p=p_{\lambda_0}$ is a slow-fast Hopf point if and only if
    \[
        Z_{\lambda_0}(F_{\lambda_0})(p ) = 0, \mathcal{G}_{\lambda_0}(p )= 0, V_{\lambda_0}(\mathcal{G}_{\lambda_0})(p )<0.
    \]
    If
    \[
        \frac{\partial}{\partial\lambda}\left.\mathcal{G}_{\lambda}(p_\lambda)\right|_{\lambda=\lambda_0}\not= 0
    \]
    (i.e.~the derivative in the left hand side is a linear map $\R^n\to \R$ of rank 1).  Then the system undergoes a supercritical slow-fast Hopf bifurcation when
    \[
        \sigma := \frac12V_{\lambda_0}^2(\mathcal{G}_{\lambda_0})(p) -V_{\lambda_0}(\mathcal{G}_{\lambda_0})(p)\mathcal{A}_{\lambda_0}
    \]
    is negative (meaning that the first Lyapunov coefficient is also negative and the limit cycle is stable when it appears).  When $\sigma>0$, it is a subcritical Hopf bifurcation.  When $\sigma=0$, the Hopf bifurcation is degenerate (sometimes also called a Bautin bifurcation).
\end{theorem}

\begin{theorem}\label{thm:bt}
    Let assumptions A1--A3 be satisfied.   The contact point $p=p_{\lambda_0}$ is a slow-fast Bogdanov-Takens point if and only if
    \[
        Z_{\lambda_0}(F_{\lambda_0})(p ) = 0, \mathcal{G}_{\lambda_0}(p )= 0, V_{\lambda_0}(\mathcal{G}_{\lambda_0})(p )=0,   V_{\lambda_0}^2(\mathcal{G}_{\lambda_0})(p )\not=0.
    \]
        If
    \[
        \frac{\partial}{\partial\lambda}\left.\left(\mathcal{G}_{\lambda}(p_\lambda),V_{\lambda}(\mathcal{G}_{\lambda})(p_\lambda)\right)\right|_{\lambda=\lambda_0}
    \]
    is a linear map $\R^n\to\R$ of rank $2$, then the system undergoes a slow-fast Bogdanov-Takens bifurcation.  The sign of $V_{\lambda_0}^2(\mathcal{G}_{\lambda_0})(p)$ determines the criticality of the Hopf curve in this bifurcation diagram.
\end{theorem}

\begin{remark}
    One could adapt Theorem~\ref{thm:criticality}   to deal with slow-fast saddle points, i.e.~contact points of singularity index $-1$; in that case $\sigma$ is related to the first saddle quantity of the hyperbolic saddle that unfolds from the contact point.  We leave this to the reader.
\end{remark}

\begin{remark}
    Higher order Lyapunov values for the slow-fast Hopf points are more difficult to compute; we prefer to leave it for potential future work.
\end{remark}

\section{Proof of the results}\label{sec:proof}

\subsection{Proof of Theorem~\ref{thm:intrinsic}}

In this theorem the $\lambda$-dependence is irrelevant so we will not maintain the parameter dependence in the notation during that part, and hence conveniently write
\[
    X=FZ+\epsilon Q + O(\epsilon^2).
\]
That having said, we stress that all operations in this part lead trivially to results that are smoothly depending on $\lambda$.\medskip

\begin{proposition}
    $\mathcal{A}$ is properly defined and is a coordinate free notion.
\end{proposition}
\begin{proof}
    Note that by assumption A3, the denominator in $\mathcal{A}$ is nonzero.  Hence $\mathcal{A}$ is properly defined if we prove that it does not depend on the choice of the factorization $F.Z$. To that end, let $(F',Z') = (cF,1/cZ)$:
    \[
    \textstyle
        Z'(F') = \frac1c Z(cF) = \frac1c Z(c) F + Z(F)
    \]
    and
    \[
    \textstyle
        (Z')^2(F')  = \frac1c Z (\frac1c Z(c) F + Z(F)) = \frac1{c^2}Z(c)Z(F) + \frac1c Z(\frac1c Z(c))F + \frac1cZ^2(F).
    \]
    At the point $p$, this evaluates to $\frac1c Z^2(F)$, since $F(p)=Z(F)(p)=0$ (assumption A1: $p\in S$ and assumption A3).  Let us apply $Z'$ once more,
    but ignoring the terms that are $O(Z(F))$ and $O(F)$ since we will only need to evaluate the expression at the point $p$, where $F(p) = Z(F)(p)=0$; only three terms remain: one from deriving the first term in $(Z')^2(F')$, and two from the last term.
    \[\textstyle
        (Z')^3(F')(p)  =  \frac{1}{c^3}Z(c)Z^2(F)(p) + \frac1c Z(\frac1c)Z^2(F)(p) + \frac1{c^2} Z^3(F)(p).
    \]
    Two additional terms cancel each other, so
    \[
      (Z')^3(F')(p)  =  \frac1{c^2} Z^3(F)(p).
    \]
    We conclude that
    \[
      \frac{(Z')^3(F')(p)}{(Z')^2(F')(p)^2}  =  \frac{Z^3(F)(p)}{Z^2(F)(p)^2}.
    \]
    Since we now know that $\mathcal{A}$ is a properly defined notion associated to the slow-fast system $X$ and since it is defined using intrinsic notions such as Lie derivatives, it is a coordinate free value.
\end{proof}

\begin{proposition}
    $\mathcal{G}$ is a properly defined function and is a coordinate free notion.
\end{proposition}
\begin{proof}
    Several points need to be discussed.  First, $Q$ is only defined properly on the critical set $S$, and even then it is only defined up to $O(Z)$.  In other words, the value for $\mathcal{G}$ must not change upon changing $Q$ to $Q + cZ$ for some function $c$.  Just observe that
    \[
        \Omega(Q+cZ,Z) = \Omega(Q,Z) + c\Omega(Z,Z) = \Omega(Q,Z).
    \]
    A second point is the invariance w.r.t.~the factorization of $FZ$; we deal with it like in the proof of the previous proposition: let $(F',Z',Q') = (cF,1/cZ,Q)$.  Then
    \[
        \Omega(Q',Z') = \Omega(Q,(1/c)Z) = (1/c)\Omega(Q,Z).
    \]
    Recall from the previous proposition that $Z'(F')=Z(F) + \frac{1}{c}Z(c)F$, so
    \[
        \nabla Z'(F') = \nabla Z(F) + \frac{1}{c}Z(c)\nabla F + F\nabla\left(\frac{1}{c}Z(c)\right),
    \]
    and $\nabla F' = \nabla (cF) = c\nabla F + F\nabla c$.  Combining these equations, we find
    \[
        \Omega (\nabla F',\nabla Z'F') = \Omega(c\nabla F,\nabla Z(F) + \frac{1}{c}Z(c)\nabla F)) + O(F) = c\Omega(\nabla F,\nabla Z(F)).
    \]
    We conclude that the value of $\mathcal{G}$ has only changed by an amount that is $O(F)$ upon changing $(F,Z,Q)$ to $(F',Z',Q')$.  Since the domain of $\mathcal{G}$ is restricted to $S=\{F=0\}$ we have proved the proposition.
\end{proof}

\begin{proposition}
    $\mathcal{G}$ is independent of the metric.
\end{proposition}
\begin{proof}
    The short answer is that $\mathcal{G}$ is multi-linearly composed from two co-variant and two contra-variant elements; we prefer to give a more elaborate proof though.
    Let $g$ be a metric on the manifold $M$.  Let $(E,F,G)$ be the coefficients of the first fundamental form, expressed in some chart in $(x,y)$ coordinates.  So with respect to the metric $g$ we have
    \[
        \langle a(q), b(q)\rangle = a(q)^T\left(\begin{array}{cc}E(q) & F(q)\\F(q) & G(q)\end{array}\right)b(q),
    \]
    where $a(q)$ and $b(q)$ are two vectors at $q$ expressed in $(x,y)$-coordinates, $T$ denotes the transpose.  We recall that the gradient of a function $\alpha$ in this metric is such  that
    \[
        \langle a(q), (\nabla f)(q)\rangle = a(f)(q),
    \]
    for all vector fields $a$.  Since in this chart $a(f)(q) = a_x(q)\frac{\partial f}{\partial x}(q) + a_y(q) \frac{\partial f}{\partial y}(q)$, we see
    that
    \[
    a(q)^T\left(\begin{array}{cc}E(q) & F(q)\\F(q) & G(q)\end{array}\right)(\nabla f)(q)=\left(\begin{array}{cc}a_x(q) & a_y(q)\end{array}\right)\left(\begin{array}{c}\frac{\partial f}{\partial x}(q)\\\frac{\partial f}{\partial y}(q)\end{array}\right).
    \]
    Since this is true for all $a$, we can apply it for $a=(1,0)$ and $a=(0,1)$ to find
    \[
        \left(\begin{array}{cc}E(q) & F(q)\end{array}\right)(\nabla f)(q)=\frac{\partial f}{\partial x}(q)
    \]
    and
    \[
        \left(\begin{array}{cc}F(q) & G(q)\end{array}\right)(\nabla f)(q)=\frac{\partial f}{\partial y}(q).
    \]
    It is a lengthy way of saying that
    \[
        (\nabla f)(q) = \left(\begin{array}{cc}E(q) & F(q)\\F(q) & G(q)\end{array}\right)^{-1}\left(\begin{array}{c}\frac{\partial f}{\partial x}(q)\\\frac{\partial f}{\partial y}(q)\end{array}\right).
    \]
    Let $\Omega$ be the area form associated to the metric.  Then it is known that $\Omega = \sqrt{EG-F^2}(du\wedge dv)$.
    So
    \[
        \Omega_q (\nabla F(q),\nabla ZF(q)) = \sqrt{EG-F^2}(q) \det\left(\nabla F(q) | \nabla ZF(q)\right).
    \]
    If $M = \left(\begin{smallmatrix}E&F\\F&G\end{smallmatrix}\right)$, then this evaluates to
    \[
        \Omega_q (\nabla F(q),\nabla ZF(q)) = \sqrt{EG-F^2}(q) \det\left(M_q^{-1}\nabla^\circ F(q) | M_q^{-1}\nabla^\circ ZF(q)\right),
    \]
    where $\nabla^\circ$ is the gradient operator in $(x,y)$-coordinates w.r.t.~the standard metric. We continue:
    \begin{align*}
        \Omega_q (\nabla F(q),\nabla ZF(q)) &= \sqrt{EG-F^2}(q) \det\left(M_q^{-1}.\left(\nabla^\circ F(q) | \nabla^\circ ZF(q)\right)\right)
            \\&= \frac{\sqrt{EG-F^2}(q)}{\det M_q} \:.\:\Omega^{\circ}_q(\nabla^\circ F(q),\nabla^\circ Z(F)(q)),
    \end{align*}
    where $\Omega^\circ$ is the area form associated to the standard metric in the $(x,y)$-chart.  Similarly,
    \[
        \Omega_q(Q(q),Z(q)) = \sqrt{EG-F^2}(q)\Omega^\circ_q(Q(q),Z(q)).
    \]
    We find finally that
    \[
        \Omega_q(Q(q),Z(q))\Omega_q (\nabla F(q),\nabla ZF(q)) = \Omega^\circ_q(Q(q),Z(q))\Omega^{\circ}_q(\nabla^\circ F(q),\nabla^\circ Z(F)(q)),
    \]
    and since this is true for all  choices of the metric $g$, we conclude that $\mathcal{G}$ does not depend on it.
\end{proof}

\subsection{An abstract normal form}

As before we prefer to leave out the dependence on $\lambda$ in the notations.  The stated results will be valid in a $\lambda$-smooth way though.

\begin{proposition}\label{prop:nf}
    Let the slow-fast system $X=FZ+\epsilon Q+O(\epsilon^2)$ satisfy Assumptions A1--A2 and let $p\in S=\{F=0\}$ be a contact point satisfying Assumption A3.
    Then there exists a regular chart of $M$ where $X$ is given by
    \[
    \begin{vf}
        \dot{x} &=& \phi(x,y,\epsilon)(y-f(x))\\
        \dot{y} &=& \epsilon \phi(x,y,\epsilon).g(x,y,\epsilon)
    \end{vf}
    \]
    The point $p$ is at the origin in this chart, and $f(0)=f'(0)=0$, $f''(0)\not=0$.  Furthermore, $\phi(0,0,0)=1$.
    We could even additionally ensure that $f(x)=x^2/2$.
\end{proposition}

\begin{remark}
    It does not automatically mean that the triplet $(F,Z,Q)$ is such that $F= y-f(x) $ in this chart, but  in view of Remark~\ref{rm:triplet} we can replace the factorization of $FZ$ in a way that $F$ indeed equals $y=f(x)$.
\end{remark}

The proof of Proposition~\ref{prop:nf} is inspired from the one in \cite{Mono}, we have added a modified proof here for the sake of completeness.  The rest of the section is devoted to proving Proposition~\ref{prop:nf}.
\begin{lemma}
    Let $u=ZF$, $v=F$.  Then $(u,v)$ are coordinates of a regular chart of $M$ bringing $p$ at the origin and such that
    \[
        \begin{vf}
        \dot{u} &=& v. Z^2F(u,v) + O(\epsilon),\\
        \dot{v} &=& vu + O(\epsilon).
        \end{vf}
    \]
\end{lemma}
\begin{proof}
    Consider any  chart in $(x,y)$ coordinates, and take the standard metric.  Since $Z^2F(p)\not=0$, at least one of the components of $\nabla ZF$ must be nonzero.  We have also assumed that the critical curve is smooth so $\nabla F$ is nowhere the zero vector.  Suppose $\nabla ZF = c\nabla F$ at the point $p$.  Then clearly also $Z(ZF))$ and $Z(F)$ should be proportional by the same constant.  Since $Z(F)(p)=0$ and $Z(ZF)(p)\not=0$, this is a contradiction.  The rest of the statement can be verified for $\epsilon=0$.  Recall that $X_0 = FZ$, so
    \[
        \begin{vf}
        \dot{u}|_{\epsilon=0} &=& X_0(ZF) = F.Z^2F = v . Z^2F(u,v)    \\
        \dot{v}|_{\epsilon=0} &=& X_0(F) = F.ZF = uv.
        \end{vf}
    \]
    This proves the result.
\end{proof}

\begin{remark}\label{rm:proofpropgen}
    In the proof, the map $q\in M\mapsto (F(q),Z(F)(q))$ is shown to be smoothly invertible.  If we for a moment consider again the parameter-dependent context and suppose that we have assumed that the conditions A1--A3 are satisfied at $\lambda=\lambda_0$, then this would imply the possibility to apply the implicit function theorem to the map $(q,\lambda) = (F_(\lambda)(q),Z_\lambda(F_\lambda)(q))$ at $(q,\lambda)=(p,\lambda_0)$.  This proves Proposition~\ref{prop:genericity}.
\end{remark}

From this chart, we can work towards a chart where the flow is rectified.  One could just refer to the flow box theorem, or we could just state that we search for $y=\alpha(u,v)$ so that $\dot{y}$ becomes $0$.  Then $\alpha$ satisfies the PDE $\alpha_u.Z^2F(u,v) + u\alpha_v=0$.  The method of characteristics guarantees a solution satisfying the boundary condition $\alpha(0,v)=v$.  We find
\[
    \dot{u} = F.\psi,\dot{y} = 0,
\]
with $\psi\not=0$.
It is not so hard to deduce, based on the intrinsic conditions A1--A3, that $F=\alpha(u,y).(y-f(u))$ for some function $f$ with $f(0)=f'(0)=0$, $f''(0)\not=0$, $\alpha\not=0$.  We then obtain
    \[
    \begin{vf}
        \dot{u} &=& \phi(u,y,\epsilon)(y- f(u))\\
        \dot{y} &=& O(\epsilon).
    \end{vf}
    \]
We rewrite the $O(\epsilon)$-terms in the $\dot{y}$-equation as $\phi(u,y,\epsilon).g(u,y,\epsilon)$.
From that moment on, we can easily find a change of coordinates $u=\theta(x)$ so that $f(\theta(x))=x^2/2$.
Using a final scaling $x=c\tilde{x},y=c^2/2\tilde{y}$, we can ensure $\phi(0,0,0)=1$.

\subsection{The skewness invariants in normal form}

As before we prefer to leave out the dependence on $\lambda$ in the notations.  It is only in the next subsection that we deal with bifurcations.

\begin{proposition}\label{prop:intrinsic}
    There exists a change of coordinates so that the slow-fast vector field  is given by
    \[
    \begin{vf}
        \dot{x} &=& \phi(x,y,\epsilon)(y-x^2/2)\\
        \dot{y} &=& \epsilon\phi(x,y,\epsilon)(g_0 + g_1x + g_2 x^2 + O(x^3) + O(y-x^2/2) + O(\epsilon)  ))
    \end{vf}
    \]
    with
    \[
        g_0 = -\mathcal{G}(p),\qquad g_1 = V(\mathcal{G})(p) + O(g_0) ,\qquad g_2 =  g_1\mathcal{A} -\frac12V^2(\mathcal{G})(p) + O(g_0).
    \]
    and
    \[
        \phi(x,0,0) = 1 -\frac{\mathcal{A}}{3}x + O(x^2).
    \]
\end{proposition}

\begin{remark}
    Let us comment on the skewness, \emph{in case of a singular contact point, eg.~with $g_0=0$}.  In the fast vector field, eg.~for $\epsilon=0$, we have $\dot{x} = (1-\frac13\mathcal{A}x + O(x^2)(y-x^2/2)$.  A change of coordinates $x=\bar x(1-\frac16\mathcal{A}\bar x) + O(\bar x^2)$ can change it into $\dot{\bar x} = (1+O(y,\epsilon))(y-f(\bar x))$, hence erasing the principle part of $\phi$.  We would
    then see that
    \[
        f(\bar x) = \frac12\bar x^2 - \frac16\mathcal{A}\bar x^3 + O(\bar x^4).
    \]
    In that sense, $\mathcal{A}$ measures the asymmetry in the critical curve up to first order.  We therefore call $\mathcal{A}$ the skewness of the fast system.  The skewness of the slow system can be seen from the equivalent vector field
    \[
    \begin{vf}
        x' &=& y-x^2/2\\
        y' &=& \epsilon (g_1x + g_2 x^2 + O(x^3) + O(y-x^2/2) + O(\epsilon)  ))
    \end{vf}
    \]

    If we then apply a change of coordinates $x = \bar x  - \frac{g_2}{g_1}\bar x^2 + O(\bar x^3)$, then the vector field changes into
    \[
    \begin{vf}
        x' &=& y-f(\bar x)\\
        y' &=& \epsilon (g_1\bar x  + O(y-f(\bar x)) + O(\epsilon)  ))
    \end{vf}
    \]
    with
    \[
        f(\bar x) = \frac12\bar x^2  - \frac{g_2}{g_1}\bar x^3 + O(\bar x^4)
    \]
    Again, we see that $g_2/g_1 = \mathcal{A} - \frac12\frac{V^2(\mathcal{G})(p)}{V(\mathcal{G})(p)}$ measures the skewness of the equivalent system.  We could attribute  this skewness to the slow-fast system, and since $\mathcal{A}$ was solely attributed to the fast system, it makes sense to say that the skewness of the slow-fast system is composed of the fast skewness $\mathcal{A}$ and the slow skewness $V^2(\mathcal{G})(p)$.
\end{remark}

\begin{proof}
We rewrite the result from Proposition~\ref{prop:nf} in a more convenient form:
\[
    \begin{vf}
        \dot{x} &=& \phi(x,y,\epsilon)(y-x^2/2)  \\
        \dot{y} &=& \epsilon\phi(x,y,\epsilon)( g (x) + (y-x^2/2)h(x,y) + O(\epsilon))
    \end{vf}
\]
where $\phi(0,0,0)=1$.
Let us identify a triplet $(F,Z,Q)$:
\[
    F = y- x^2/2,\quad Z = \phi(x,y,0)\frac{\partial}{\partial x},
\]
and
\[
    Q = \phi_1(x,y)(y- x^2/2)\frac{\partial}{\partial x} + \phi(x,y,0)(g(x) + (y-x^2/2)h(x,y))\frac{\partial}{\partial y}
\]
where $\phi_1$ is the Taylor coefficient with $\epsilon^1$ of $\phi(x,y,\epsilon)$. The nice thing is that $Q$ is only defined at $F=0$, so we might as well use
\[
    Q = \phi(x,y,0)g(x)\frac{\partial}{\partial y}.
\]
\begin{remark}\label{rm:slowvf}
    $Q$ is not the slow vector field; in fact the slow vector field is the unique field $\tilde{Q} = Q + \alpha Z$ tangent to $S$, i.e.~for which $\tilde{Q}(F)=0$.  It is not so hard that this would entail choosing
    \[
        \tilde{Q} = Q - \frac{Q(F)}{Z(F)}Z = \phi(x,y,0)\frac{g(x)}{f'(x)}\partial_x,
    \]
    where $\partial_x := \left(x\frac{\partial}{\partial y} +  \frac{\partial}{\partial x}\right)$.  Replacing $Q$ by $\tilde{Q}$ in the sequel does not change anything though.
\end{remark}
As proven in Theorem~\ref{thm:intrinsic}, we are allowed to do the computations in any coordinate system using any metric and any choice of factoring $F.Z$.
We find
\[
    \Omega(Q,Z) = \left|\begin{array}{cc} 0 & \phi  \\ \phi g & 0 \end{array}\right| = -\phi^2g(x).
\]
Also, $ZF = -x\phi  $, so
\[
    \Omega(\nabla F,\nabla ZF) =  \left|\begin{array}{cc}  -x   &  \frac{\partial}{\partial x}(-x\phi  ) \\   1 & \frac{\partial}{\partial y}(-x\phi )\end{array}\right| = \partial_x(x\phi  ),
\]
where $\partial_x$ is like in Remark~\ref{rm:slowvf}.
So
\[
    \mathcal{G} = -\phi^2\partial_x(x\phi  ) g.
\]
Let us now compute the invariant $\mathcal{A}$:
\begin{align*}
    ZF &= -x\phi\\ Z^2F &= -\phi\left(\phi + x\frac{\partial\phi}{\partial x}\right)  \\ Z^3F &=  -\phi^2 \left(2\frac{\partial\phi}{\partial x} + x\frac{\partial^2\phi}{\partial x^2}\right)  -  \phi\frac{\partial\phi}{\partial x}\left(\phi + x\frac{\partial\phi}{\partial x}\right).
\end{align*}
\begin{remark}
    We use $\frac{\partial}{\partial x}$ here, not $\partial_x$.  This subtle difference is actually the reason why the invariant $\mathcal{A}$ is restricted to a point: $\frac{\partial}{\partial x}$ is only intrinsic up to $O(x^2)$.
\end{remark}
At the origin: $ZF(0)=0$, $Z^2F(0) = -\phi^2(0)=-1$, $Z^3F(0)=-3 \frac{\partial \phi}{\partial x}(0)$.  We conclude
\[
    \mathcal{A} = -3\frac{\partial \phi}{\partial x}(0).
\]
(Recall that we have taken $\phi(0)=1$.)

Let us now turn our attention to $\mathcal{G}$. To that  end, observe that
\[
    V = \frac{1}{\partial_x (-x\phi)}\partial_x,
\]
(Indeed: $V(F)=0$ and $V(ZF)=V(-x\phi)=1$, and these were the two governing conditions for $V$.)
Notice now that $\partial_x( x\phi) = \phi + O(x)$, so
\[
    \mathcal{G} = -\phi^2\partial_x(x\phi  ) g\implies \mathcal{G}(0) = -g(0) = -g_0.
\]
The remaining coefficients $g_1$ and $g_2$ will be determined up to $O(g_0)$, so we can and will assume from now on until the end of this proof that $g_0=0$.
To compute $g'(0)$, we derive $\mathcal{G}$:
\[
    V(\mathcal{G}) = \partial_x (\phi^2g) + \phi^2g \frac{\partial^2_x(x \phi )}{\partial_x(x\phi)}
\]
Since $g(0)=0$, we have
\[
    V(\mathcal{G})(0) = \phi^2(0,0,0)g'(0) = g'(0).
\]
Let us finally derive one more time, but disregarding the $O(g)$ terms since we do not keep track of them.  In that respect, $V = -\partial_x$ since $\partial_x(x\phi )|_0 = \phi(0)=1$:
\begin{align*}
    V^2(\mathcal{G})(0) &= -\left.\partial_x^2 (\phi^2g) -  \phi^2\partial_xg \frac{\partial^2_x(x \phi )}{\partial_x(x\phi)}\right|_0   \\
            &= -\left.\partial_x(\partial_x\phi^2 g)- \partial_x(\phi^2\partial_x g) - g'(0) \frac{\partial^2_x(x \phi )}{\partial_x(x\phi)}\right|_0  \\
            &= -\left.2\partial_x\phi^2g'(0) - \phi^2g''(0) - g'(0)\frac{\partial_x(x\partial_x\phi + \phi)}{1}\right|_0 \\
            &= -4\phi(0)\frac{\partial\phi}{\partial x}(0)g'(0) - g''(0) -  2g'(0)\frac{\partial\phi}{\partial x}(0)  \\
            &= -6\frac{\partial\phi}{\partial x}(0)g'(0) - g''(0)  \\
            &= 2\mathcal{A}g'(0) - g''(0)  .
\end{align*}
So the theorem follows since $V^2(\mathcal{G})(0) = 2\mathcal{A}g_1 - 2g_2 + O(g(0))$.
\end{proof}

\subsection{Proof of Theorem~\ref{thm:criticality}}

We apply Proposition~\ref{prop:intrinsic} so that we can work with the system
    \[
    \begin{vf}
        \dot{x} &=& \phi(x,y,\epsilon,\lambda)(y-x^2/2)\\
        \dot{y} &=& \epsilon\phi(x,y,\epsilon,\lambda)(g_0(\lambda) + g_1(\lambda) x   + g_2(\lambda) x^2
     \\&&\qquad\qquad\qquad + O(x^3) + O(y-x^2/2) + O(\epsilon)  ))
    \end{vf}
    \]
to prove the claim.  Theorem~\ref{thm:intrinsic} justifies that we may choose any coordinate system.  We have
\[
    g_0(\lambda) = -\mathcal{G}_\lambda (p_\lambda),
\]
and
\[
        g_1(\lambda) = V_\lambda(\mathcal{G}_\lambda)(p_\lambda) + O(g_0),
        \quad g_2(\lambda) =  g_1(\lambda)\mathcal{A}(\lambda) -\frac12V_\lambda^2(\mathcal{G}_\lambda)(p_\lambda) + O(g_0).
\]
Since $g_0(\lambda_0)=0$ by assumption, the $O(g_0)$ terms are actually $O(\|\lambda-\lambda_0\|)$.\medskip

Since the map $\lambda\to \frac{\partial g_0}{\partial\lambda}$ is nonsingular, there exists (by the submersion theorem a smooth reparametrization
$\lambda = \lambda(\mu)$ so that
\[
    g_0(\lambda(\mu)) = \mu_1,\qquad \forall \mu=(\mu_1,\dots,\mu_n)\approx (0,\dots,0).
\]
We consider now the \emph{associated} and equivalent slow-fast vector field
    \begin{equation}\label{eq:normalhopf}
    \begin{vf}
        \dot{x} &=& y-x^2/2 \\
        \dot{y} &=& \epsilon (a + g_1(\lambda(\mu)) x   + g_2(\lambda(\mu)) x^2
     \\&&\qquad\qquad\qquad + O(x^3) + O(y-x^2/2) + O(\epsilon)  )).
    \end{vf}
    \end{equation}
We have replaced $\lambda$ by its reparameterized version $\lambda(\mu)$, except for the constant   term in $g$ where we have introduced the so-called breaking parameter $a$.
This has the benefit that it is now an $(a,\mu)$-family of slow-fast vector fields where $a$ is the breaking parameter that does not appear in other places; uniformly in $\mu$ and for $a=0$ the origin is a slow-fast Hopf point now.   Standard theory on canards now shows the presence of a Hopf curve $a=a_H(\mu,\epsilon) = O(\epsilon)$ along which a Hopf bifurcation occurs, growing into a canard explosion along canard curves $a=a_C(\mu,\epsilon,y_0)  = O(\epsilon)$, where $y_0$ is the ``height'' of the canard.\medskip

If we now replace $a$ again by $\mu_1$, we have an implicit equation for the Hopf and canard curves
\[
    \mu_1 = a_H(\mu,\epsilon),\qquad \mu_1 = a_C(\mu,\epsilon),
\]
which one can easily solve by the Implicit Function Theorem.  It shows that the bifurcation diagram in the $(a,\mu,\epsilon)$-space transforms to a bifurcation diagram in the $(\mu,\epsilon)$-space and hence also in the $(\lambda,\epsilon)$-space.\medskip

Concerning the criticality: there is an elementary trick to deduce the criticality of the Hopf bifurcation in (\ref{eq:normalhopf}).  We use formal analysis to determine the canard curve $a=a_C(\mu,\epsilon)$.  To that end, we compute to some extent the slow curve, along $a=\epsilon a_0 + O(\epsilon^2)$ and find
\[
    y = \frac{x^2}{2} + \epsilon(g_1 + g_2 x + O(x^2)) + \epsilon^2\left(\frac{hg_1 - g_1g_2 + a_0}{x} + O(1)\right) + O(\epsilon^2).
\]
Here $h$ is the constant coefficient of the $O(y-x^2/2)$ term in $\dot{y}$, which we have not named before (but that will cancel out anyway).
It is well-known that the Taylor coefficients of the canard curve can be found recursively, by avoiding poles in the expression of the formal slow curve.  In other words we choose
\[
    a_0 = g_1g_2 - hg_1.
\]
Next step is  to see that the singular point has coordinates $(x,y) = \epsilon(-a_0/g_1,0) + O(\epsilon^2)$, and if we compute the trace of the jacobian matrix at that point we find
\[
    T = \frac{a_0}{g_1 } + h = g_2.
\]
When $T>0$, the Hopf point is unstable at the moment of the canard explosion, implying that the limit cycle surround the Hopf point is stable.  In other words, $g_2>0$ corresponds to a supercritical Hopf bifurcation.  Since
\[
    g_2|_{\mu=0} = g_2|_{\lambda=\lambda_0} = V_{\lambda_0}(\mathcal{G}_{\lambda_0})(p)\mathcal{A}_{\lambda_0} - \frac12V_{\lambda_0}^2(\mathcal{G}_{\lambda_0})(p) = -\sigma,
\]
we have proved Theorem~\ref{thm:criticality}.

\subsection{Proof of Theorem~\ref{thm:bt}}

Like before, using Proposition~\ref{prop:intrinsic} and Theorem~\ref{thm:intrinsic}, combined with the fact that in studying bifurcations it is allowed to rescale time, we will work with
\[    \begin{vf}
        \dot{x} &=& y-x^2/2\\
        \dot{y} &=& \epsilon(g_0 + g_1x + g_2 x^2 + O(x^3) + O(y-x^2/2) + O(\epsilon)  ))
    \end{vf}
\]
Using this normal form, and under the conditions of Theorem~\ref{thm:bt}, we can directly apply the main (unnumbered) theorem of \cite{sfbt} to conclude the results.

\section{Examples}\label{sec:examples}

\subsection{Van der Pol}

It goes without saying that the computations for the Van der Pol system
\[
    \begin{vf}
        \dot{x} &=& y-\frac12x^2-\frac13x^3\\
        \dot{y}&=&\epsilon(\lambda-x)
    \end{vf}
\]
are elementary: $F=y-\frac12x^2-\frac13x^3$, $Z=\frac{\partial}{\partial x}$ and $Q=(\lambda-x)\frac{\partial}{\partial y}$.  We compute $ZF = -x-x^2$, so
\[
    \Omega(Q,Z) =x-\lambda,\qquad \Omega(\nabla F,\nabla ZF) = (1+2x),\qquad \mathcal{G} = (1+2x)(x-\lambda).
\]
The vector field $V=\frac{-1}{2x+1}\frac{\partial}{\partial x} - \frac{x^2+x}{2x+1}\frac{\partial}{\partial y}$ is such that $V(F)=0$, $V(ZF)=1$, as requested in the presented theory.  We see that $\mathcal{G}(0)=-\lambda$,
\[
    V(\mathcal{G})|_{\lambda=0} = -\frac{4x+1}{2x+1}\implies V(\mathcal{G}(0) < 0.
\]
It means that the origin is indeed a slow-fast Hopf point.  For the criticality, we check
\[
    V^2(\mathcal{G})|_{\lambda=0} = V(-1-2x + O(x^2)) = V(-2x) + O(x) = -2 + O(x).
\]
On the other hand, $ZZF = -1-2x$, $ZZZF=-2$, so $\mathcal{A}=-2$.  We infer that $\sigma  <0$, which indicates a supercritical Hopf bifurcation, as is well-known (see \cite{memoir} for example).

\subsection{Slow-fast Hopf bifurcation in a non-standard form}

Consider
\[
    \begin{vf}
        \dot{x} &=& y(\delta-y) \\
        \dot{y} &=& (-x+\alpha y)(\delta- y) - \epsilon(\beta-\gamma x).
    \end{vf}
\]
The model is a so-called a two-stroke oscillator and is discussed in detail in section 6.1.4 of \cite{WECH}.  The parameters $\alpha,\beta,\delta,\gamma$ are all strictly positive.  We do the computations:

\begin{enumerate}
\item   
    Identification of $(F,Z,Q)$: $F=\delta-y$, $Z=y\frac{\partial}{\partial x} + (-x+\alpha y)\frac{\partial}{\partial y}$ and $Q = -(\beta-\gamma x)\frac{\partial}{\partial y}$.  Assumption A1 is clearly satisfied.  Note that Assumption A2 stipulates that $Z$ should be nonzero, whereas there  is clearly a singular point at the origin.  However, since our results are local and since the origin does not lie on the critical curve, we can restrict the phase space to an open set of $\R^2$ where Assumption A2 is satisfied.
\item   
    Determining contact points (eg.~solving $F=ZF=0$): since $ZF=x-\alpha y$, there is only one contact point $p=(\alpha\delta,\delta)$.  Observe that $ZZF=y+\alpha (x-\alpha y)$, so $ZZF(p)=\delta\not=0$.  It is therefore a generic contact point satisfying Assumptions A3.
\item
    We compute $ZZZF = \alpha y + (\alpha y-x)(1-\alpha^2)$ which evaluates to $\alpha\delta$ at $p$.  It follows that the first invariant is given by $\mathcal{A} =\frac{\alpha}{\delta}$.
\item
    For the second invariant, $\mathcal{G}$, we compute $\nabla F=(0,-1)$ and $\nabla ZF = (1,-\alpha)$.  So
    \[
        \mathcal{G} = \Omega(Q,Z)\Omega(\nabla F,\nabla ZF) = \left|\begin{array}{cc} 0 & y \\ -(\beta-\gamma x) & -x+\alpha y
        \end{array}\right|.\:\left|\begin{array}{cc} 0 & 1\\ -1 & -\alpha
        \end{array}\right|.
    \]
    We obtain $\mathcal{G} = y(\beta-\gamma x)$.
\item
    We find $\mathcal{G}(p) = \delta(\beta-\alpha\gamma\delta)$, so for $\beta = \beta_0 := \alpha\gamma\delta$ the point $p$ is a singular contact point (as is seen in the analysis in \cite{WECH}).
\item
    In order to determine the nature of this contact point, we examine the derivatives of $\mathcal{G}$.  This is easy since $V=\frac{\partial}{\partial x}$ in this case: $V(F)=0$ and $V(ZF)=1$.  So $V(\mathcal{G}) = -\gamma y$ and $V^2(\mathcal{G}) = 0$.  We see that the conditions for a slow-fast Hopf point are satsified, since $V(\mathcal{G})(p)<0$, and we easily find
    \[
        \sigma = \frac12V^2(\mathcal{G})(p) - V(\mathcal{G})(p)\mathcal{A} = \gamma \alpha.
    \]
    It is strictly positive; we therefore confirm the determination of a subcritical Hopf bifurcation that was obtained in \cite{WECH}.
\end{enumerate}

\subsection{Revisiting criticality conditions in \cite{krupaszmolyan}}

In Section 3.2 of \cite{krupaszmolyan}, the authors study the following system
\[
    \begin{vf}
        \dot{x} &=& -y h_1(x,y,\lambda,\epsilon) + x^2 h_2(x,y,\lambda,\epsilon) + \epsilon h_3(x,y,\lambda,\epsilon)   \\
        \dot{y} &=& \epsilon( xh_4(x,y,\lambda,\epsilon) - \lambda h_5(x,y,\lambda,\epsilon) + yh_6(x,y,\lambda,\epsilon))
    \end{vf}
\]
where
\[
    h_3(x,y,\lambda,\epsilon) = O(x,y,\lambda,\epsilon),
\]
and
\[
    h_j(x,y,\lambda,\epsilon) = 1 + O(x,y,\lambda,\epsilon),\qquad j=  1,2,4,5.
\]
We examine the contact point $p=(0,0)$, and extract an $(F,Z,Q)$-triplet:
\[
    F = -yh_{10} + x^2h_{20},\qquad Z = \frac{\partial}{\partial x}
\]
and
\[
    Q = q_1\frac{\partial}{\partial x} + (xh_{40} - \lambda h_{50} + yh_{60})\frac{\partial}{\partial y}.
\]
In these formulas, $h_{j0}(x,y,\lambda) := h_j(x,y,\lambda,0)$ for $j=1,2,4,5,6$.  The function $q_1=q_1(x,y,\lambda)$ is $O(x,y,\lambda)$.
In the rest of this section we adopt the notation that an extra subscript means derivation w.r.t.~a variable, for example $h_{10x} = \frac{\partial h_{10}}{\partial x}$.
We compute
\[
    ZF =  -yh_{10x} + 2xh_{20} +  x^2h_{20x},\quad
    ZZF =  2h_{20} + 4xh_{20x} + O(x^2) + O(y),
    \]and\[
    ZZZF = 6h_{20x} + O(x) + O(y).
\]
So $ZF(0)=0$, $ZZF(0)=2h_{20}(0)=2$, $ZZZF(0)= 6a_3$, where we adopt the notation from \cite{krupaszmolyan} by putting $a_3 := \frac{\partial h_2}{\partial x}(0)$.
The first invariant is hence given by
\[
    \mathcal{A} = \frac32 a_3.
\]
Next, observe that
\[
    \Omega(Q,Z) = \left|\begin{array}{cc}
          *    & 1 \\
         xh_{40} - \lambda h_{50} + yh_{60} &0
    \end{array}\right| = -xh_{40} +\lambda h_{50} - yh_{60}.
\]
We compute the expression up to $O(x^3,xy,y^2)$.  So $h_{40} = 1+    h_{40x}x + O(x^2) = 1 + a_4x + O(x^2)$, adopting again the notation $a_4 = \frac{\partial h_4}{\partial x}(0)$ from \cite{krupaszmolyan}.  Write also $h_{60} = a_5 + O(x,y)$.  It means
\[
    \Omega(Q,Z) = -x -a_4x^2 + \lambda h_{50} - a_5y + O(x^3,xy,y^2).
\]
We also compute
\[
    \Omega(\nabla F,\nabla ZF) = \left|\begin{array}{cc}
         2xh_{20} + O(x^2,y) & 2h_{20} + 4x h_{20x} + O(x^2,y)
        \\
        -h_{10} + O(x^2,y) & -h_{10x} + 2xh_{20y} + O(x^2,y)
    \end{array}\right|
\]
which evaluates to
\[
    2h_{10}h_{20}-2xh_{10x}h_{20}  +4 xh_{10}h_{20x} + O(x^2,y).
\]
We replace $h_{10}$ by $1 + xh_{10x} + O(x^2)$ and $h_{20}$ by $1+xh_{20x} + O(x^2)$.  Further notations introduced in \cite{krupaszmolyan} were $a_2 := \frac{\partial h_1}{\partial x}(0)$ so
\begin{align*}
\Omega(\nabla F,\nabla ZF) &= 2(1+a_2x)(1+a_3x)-2a_2x(1+a_3x) + 4a_3x + O(x^2) + O(y)\\
    &=
        2 +6a_3x  + O(x^2,y).
\end{align*}
As a consequence we find the second invariant
\begin{align*}
    \mathcal{G} &= (-x -a_4x^2 + \lambda h_{50} - a_5y + O(x^3,xy,y^2))(2 +6 a_3x  + O(x^2,y)) \\
                 &= 2\lambda(h_{50}+O(x,y,\lambda)) + (-x -a_4x^2  - a_5y)(2 +6a_3x) + O(x^3,xy,y^2) \\
                 &= 2\lambda(h_{50}+O(x,y,\lambda)) -2 (x + a_4x^2  + a_5y)  - 6a_3x^2 + O(x^3,xy,y^2).
\end{align*}
Clearly, $\mathcal{G}(0)=2h_{50}\lambda$, so the origin is a slow-fast Hopf point as soon as we prove that $V(\mathcal{G})(0)<0$.  We still need to determine $V$.  We know that it should be tangent to $(\nabla F)^\perp$ (eg.~perpendicular to $\nabla F$), and that $V(ZF)=1$.  We have
\[
    ZF = -yh_{10x} + 2xh_{20} +  x^2h_{20x} = -a_2y + 2x + 3a_3x^2   + O(x^3,xy,y^2),
\]
and
\begin{align*}
    (\nabla F)^\perp &= \nabla\left[-y  + x^2(1+a_3x) + O(x^3,xy,y^2)\right]^\perp\\
            &= -\left(2x + O(x^2,y)\right)\frac{\partial}{\partial y} + \left(-1 - a_2x+ O(x^2,y)\right)\frac{\partial}{\partial x}
\end{align*}
so
\begin{align*}
    \nabla F(ZF) &=  -(2 + 6a_3x) + O(x^2,y).
\end{align*}
It means that
\[
    V = \left(x + O(x^2,y)\right)\frac{\partial}{\partial y} + \left(\frac12 + \frac12( a_2-3a_3)x + O(x^2,y) \right)\frac{\partial}{\partial x}.
\]
We are now ready to compute $V(\mathcal{G})$ (for $\lambda=0$):
\[
     V(\mathcal{G}) =  \textstyle -1 +(  -2a_4 -2a_5 -a_2     -3a_3 )x + O(x^2,y)
\]
So at the origin, this value is negative, which indeed indicates a slow-fast Hopf point.  We finally compute
\[
    V^2(\mathcal{G}) = -\frac12( 2a_5  +  2a_4  + 3a_3   + a_2       ) + O(x,y),
\]
which shows that the criticality is governed by
\[
    \sigma  = \frac12
    V^2(\mathcal{G}) - V(\mathcal{G})\mathcal{A} = -\frac12(2a_5  +  2a_4  - 3a_3   + a_2 ).
\]
This value is proportional to the value $A$ in \cite{krupaszmolyan} so the results agree with their results.

\begin{remark}
    Clearly this application demonstrates that one still needs to do a bit of computing.  The benefit of doing the computations this way over putting the system in normal form is that this procedure can be implemented in a computer program and needs little or no adaptation to the specific application in mind.   A computer program with the triplet $(F,Z,Q)$ as input can easily generate $\mathcal{A}$, $\mathcal{G}$, $\sigma$ as output.  In fact, all computations done in this paper have been cross-checked using such a program.
\end{remark}

\section*{Acknowledgements}
 
 This work was supported by the bilateral research cooperation fund of the Research Foundation Flanders (FWO) and the National Foundation for Science and Technology (NAFOSTED) under Grant No G0E6618N.

\bibliography{singbif}

\end{document}